\documentclass[final,hidelinks,onefignum,onetabnum]{siamart220329}

\usepackage{amsmath}
\usepackage{amssymb}

\usepackage{xcolor}
\usepackage{hyperref}
\hypersetup{colorlinks=true, linkcolor=blue, citecolor=blue}
\usepackage{subcaption}

\crefname{assumption}{Assumption}{Assumptions}
\crefname{method}{MNN}{Methods}
\crefformat{method}{#2MNN#1#3}
\crefmultiformat{method}{#2MNN#1#3}{~and~#2MNN#1#3}{}{}
\crefrangeformat{method}{A#3#1#4 to~#5#2#6}
\crefname{section}{Section}{Sections}

\newcommand{\R}{\mathbb{R}}

\newcommand\restr[2]{\ensuremath{\left.#1\right\vert_{#2}}}

\newtheorem{remark}{Remark}
\newtheorem{assumption}{Assumption}
\newtheorem{example}{Example}
\newtheorem{method}{Method}%

\headers{Modified Neumann--Neumann methods}{E. Engström and E. Hansen}

\begin{document}
	\title{Modified Neumann--Neumann methods for semi- and quasilinear elliptic equations
\thanks{Version of  \today. \funding{This work was supported by the Swedish Research Council under the grant 2019--
05396.}}}

\author{Emil Engstr\"om\thanks{Centre for Mathematical Sciences, Lund University,
P.O.\ Box 118, SE-22100 Lund, Sweden, (\email{emil.engstrom@math.lth.se}).}
\and Eskil Hansen\thanks{Centre for Mathematical Sciences, Lund University,
P.O.\ Box 118, SE-22100 Lund, Sweden, (\email{eskil.hansen@math.lth.se}).}}
	
\maketitle

\begin{abstract}
The Neumann--Neumann method is a commonly employed domain decomposition method for linear elliptic equations. However, the method exhibits slow convergence when applied to semilinear equations and does not seem to converge at all for certain quasilinear equations. We therefore propose two modified Neumann--Neumann methods that have better convergence properties and require less computations. We provide numerical results that show the advantages of these methods when applied to both semilinear and quasilinear equations. We also prove linear convergence with mesh-independent error reduction under certain assumptions on the equation. The analysis is carried out on general Lipschitz domains and relies on the theory of nonlinear Steklov--Poincaré operators.
\end{abstract}
    

  

\section{Introduction}  
We consider a bounded Lipschitz domain $\Omega\subset\R^d$ for $d=2, 3$, with boundary $\partial\Omega$, together with the quasilinear elliptic equation with homogeneous Dirichlet boundary conditions
\begin{equation}\label{eq:strong}
    \left\{
         \begin{aligned}
                -\nabla\cdot\alpha(x, u, \nabla u)+\beta(x, u, \nabla u)&=f  & &\text{in }\Omega,\\
                 u&=0 & &\text{on }\partial\Omega.
            \end{aligned}
    \right.
\end{equation}
We will refer to the equation as semilinear if $\alpha(x, y, z) = \tilde{\alpha}(x)z$. A common approach to generate numerical schemes for elliptic equations that can be implemented in parallel and only rely on local communication is to employ nonoverlapping domain decomposition methods. 

The performance and convergence properties of these methods have been extensively studied for linear elliptic equations, as surveyed in~\cite{quarteroni,Widlund}, but less is known for the methods applied to quasilinear elliptic equations. Convergence results for overlapping Schwarz methods have been derived in~\cite{dryja97,lions2,tai98,tai02}. Note that~\cite{tai02} even considers degenerate quasilinear equations, i.e., when $\alpha(x, y, z) = \tilde{\alpha}(x, y, z)z$ and $\tilde{\alpha}(x, y, z)=0$ for some $y\neq 0$. For nonoverlapping decompositions we are only aware of a few studies that derive rigorous convergence results. In~\cite{berninger11} the convergence of the Dirichlet--Neumann and Robin--Robin methods are analyzed for a one-dimensional quasilinear elliptic equation, and in our own paper~\cite{ehee22} we prove convergence without order for the Robin--Robin method applied to degenerate quasilinear equations with a $p$-structure.  Moreover, a modified Dirichlet--Neumann method for nonlinear equations is studied in~\cite{gander23}, and convergence is derived for one-dimensional problems. There are also some studies relating to quasilinear equations and decompositions with cross points~\cite{berninger14,gander20,gander23preprint,schreiber}, but without convergence results.
We are unaware of any linear convergence results for nonoverlapping decompositions applied to quasilinear, or even semilinear, elliptic problems on general Lipschitz domains.  

The basic nonoverlapping decompositions methods are the already mentioned Dirichlet--Neumann, Neumann--Neumann, and Robin--Robin methods. In the context of linear elliptic equations the Neumann--Neumann method has better convergence properties on non-symmetric domain decompositions compared to the Dirichlet--Neumann method. Furthermore, for linear equations and after a finite element discretization one can prove that the Neumann--Neumann method converges linearly with an error reduction constant $C$ that is uniformly bounded with respect to the mesh width $h$, whereas the Robin--Robin method obtains a deteriorating constant of the form $C=1-\mathcal{O}(\sqrt{h})$. With these considerations in mind, the Neumann--Neumann method is a natural starting point for developing linearly convergent methods for more general elliptic equations. 

In order to introduce the methods, we decompose the domain $\Omega$ into two subdomains $\Omega_1$ and $\Omega_2$ such that 
\begin{displaymath}
\overline{\Omega}=\overline{\Omega}_1\cup\overline{\Omega}_2,\quad \Omega_1\cap\Omega_2=\emptyset\quad\text{and}\quad\Gamma=(\partial\Omega_1\cap\partial\Omega_2)\setminus\partial\Omega,
\end{displaymath}
where $\partial\Omega_i$ denotes the boundary of $\Omega_i$ and $\Gamma$ is the interface separating the subdomains. The subdomains can also be replaced by two families of nonadjacent subdomains, which is required in order to obtain a parallel method. Two possible decompositions are illustrated in~\cref{fig:domain}.
\begin{figure}
\centering
\includegraphics[width=0.4\linewidth]{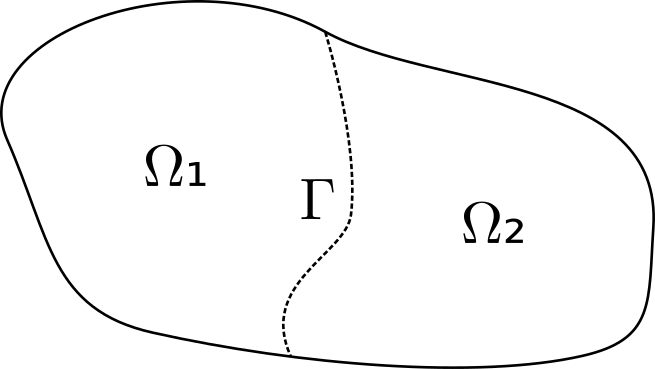}\qquad
\includegraphics[width=0.4\linewidth]{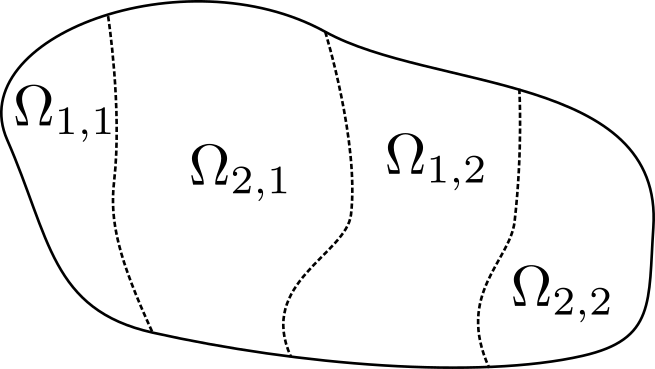}
\caption{\emph{An example of a domain decomposition with two subdomains $\Omega_1, \Omega_2$ (left), and one with a decomposition into two families of subdomains $\{\Omega_{1,\ell}\}$, $\{\Omega_{2,\ell}\}$ (right).}}
\label{fig:domain}
\end{figure}%

The Neumann--Neumann method for the nonlinear equation~\cref{eq:strong} is then as follows. Let $s_1,s_2>0$ be method parameters and $\eta^0$ an initial guess on $\Gamma$. For $n=0,1,2,\ldots,$ find $(u_1^{n+1}, u_2^{n+1}, w_1^{n+1}, w_2^{n+1})$ such that
    \begin{equation}\label{eq:nnstrong}
    	\left\{
    	\begin{aligned}
    		-\nabla\cdot\alpha(x, u^{n+1}_i, \nabla u^{n+1}_i)+\beta(x, u^{n+1}_i, \nabla u^{n+1}_i)&=f  & &\text{in }\Omega_i,\\
    		u^{n+1}_i&=0 & &\text{on }\partial\Omega_i\setminus\Gamma,\\
    		u^{n+1}_i &= \eta^n & &\text{on }\Gamma, \text{ for } i=1,2,\\[5pt]
    		-\nabla\cdot\alpha(x, w^{n+1}_i, \nabla w^{n+1}_i)+\beta(x, w^{n+1}_i, \nabla w^{n+1}_i)&=0  & &\text{in }\Omega_i,\\
    		w^{n+1}_i&=0 & &\text{on }\partial\Omega_i\setminus\Gamma,\\
    		\alpha(x, w^{n+1}_i, \nabla w^{n+1}_i)\cdot\nu_1 &=  && \\
            \alpha(x, u_1^{n+1}, \nabla u_1^{n+1})\cdot\nu_1-\alpha(x, u_2^{n+1}, \nabla  u_2^{n+1})&\cdot\nu_1 &&\text{on }\Gamma,\text{ for } i=1,2,
    	\end{aligned}
    	\right.
    \end{equation}
with $\eta^{n+1}=\eta^n-s_1\restr{w_1^{n+1}}{\Gamma}-s_2\restr{w_2^{n+1}}{\Gamma}$. Here, $\nu_1$ denote the outwards pointing unit normal on $\partial\Omega_1$, and the computed quantities $(u_1^n, u_2^n,\eta^n)$ approximate $(u|_{\Omega_1}, u|_{\Omega_2}, u|_{\Gamma})$, respectively. We will refer to the last three equalities in~\cref{eq:nnstrong} as the auxiliary problem. 

However, the Neumann--Neumann method requires the solution of four nonlinear problems in each iteration, which is inefficient even for semilinear elliptic equations. The method also displays surprisingly poor convergence properties for degenerate quasilinear equations. For example consider the semilinear equation 
\begin{equation}\label{eq:semi}
-\Delta u+|u|u=f
\end{equation}
and the $p$-Laplace equation
\begin{equation}\label{eq:3lap}
 -\nabla\cdot(|\nabla u|^{p-2}\nabla u)+u=f,
\end{equation}
both given on the unit square $\Omega=(0,1)^2$ with a decomposition into two $L$-shaped subdomains $\Omega_1$ and $\Omega_2$. The Neumann--Neumann method then requires a significantly larger amount of equation solves compared to other methods for the semilinear equation; see~\cref{fig:sub1}. Furthermore, it does not even generate a decreasing error with respect to the number of iterations for the $p$-Laplace equation; see~\cref{fig:sub2}. The details of the numerical experiments are given in~\cref{sec:numexp}.

In order to obtain methods that require less solutions of nonlinear equations, we propose two modified Neumann--Neumann methods based on replacing the auxiliary problem with a linear one, i.e., we employ a linear preconditioner. For the first modified method we replace the auxiliary problem by 
\begin{equation}\label{eq:mnn1strong}
    	\left\{
    	\begin{aligned}
    		-\Delta w^{n+1}_i&=0  & &\text{in }\Omega_i,\\
    		w^{n+1}_i&=0 & &\text{on }\partial\Omega_i\setminus\Gamma,\\
    		\nabla w^{n+1}_i\cdot\nu_1 &= \\
            \alpha(x, u_1^{n+1}, \nabla u_1^{n+1})\cdot\nu_1-\alpha(x, u_2^{n+1}, \nabla u_2^{n+1})&\cdot\nu_1 & &\text{on }\Gamma,\text{ for } i=1,2.
    	\end{aligned}
    	\right.
\end{equation}
We also consider a second method where the auxiliary problem is iteration-dependent, in particular via the linearization of the auxiliary problem. For notational simplicity, we assume that $\alpha(x, y, z)=\alpha(x, z)$ and $\beta(x, y, z)=\beta(x, y)$. We then introduce $J_\alpha$ and $J_\beta$, the Jacobians of $\alpha$ and $\beta$ with respect to $z$ and $y$, respectively. The second modified method is then given by replacing the auxiliary problem with
 \begin{equation}\label{eq:mnn2strong}
	\left\{
	\begin{aligned}
		-\nabla\cdot \bigl(J_\alpha(x, \nabla u_i^{n+1})\nabla w_i\bigr)+J_\beta(x, u_i^{n+1})w_i&=0  & &\text{in }\Omega_i,\\
		w^{n+1}_i&=0 & &\text{on }\partial\Omega_i\setminus\Gamma,\\
		J_\alpha(x, \nabla u^{n+1}_i)\nabla w^{n+1}_i\cdot\nu_1 &=  \\
        \alpha(x, \nabla u_1^{n+1})\cdot\nu_1-\alpha(x, \nabla u_2^{n+1})&\cdot\nu_1 & &\text{on }\Gamma,\text{ for } i=1,2.
	\end{aligned}
	\right.
\end{equation}

The efficiency gains of the modified Neumann--Neumann methods over the standard Neumann--Neumann method are illustrated in~\cref{fig:sub1} for our semilinear example. The modified methods even perform well for the degenerate $p$-Laplace equation; see~\cref{fig:sub2}. 
\begin{figure}
  \centering
\begin{subfigure}{.5\textwidth}
  \centering
  \includegraphics[width=1\linewidth]{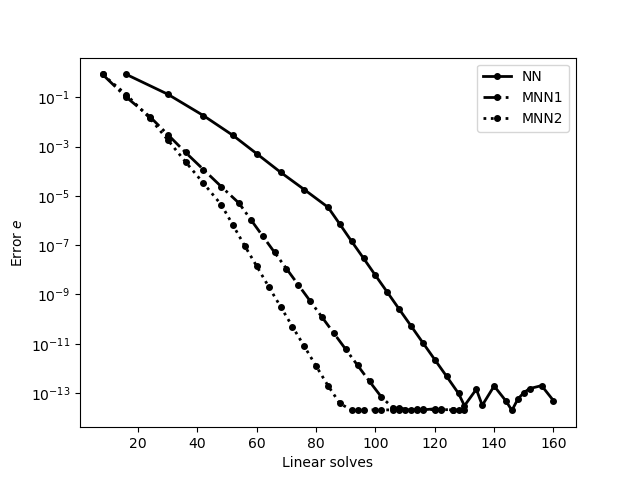}
  \caption{}
  \label{fig:sub1}
\end{subfigure}%
\begin{subfigure}{.5\textwidth}
  \centering
  \includegraphics[width=1\linewidth]{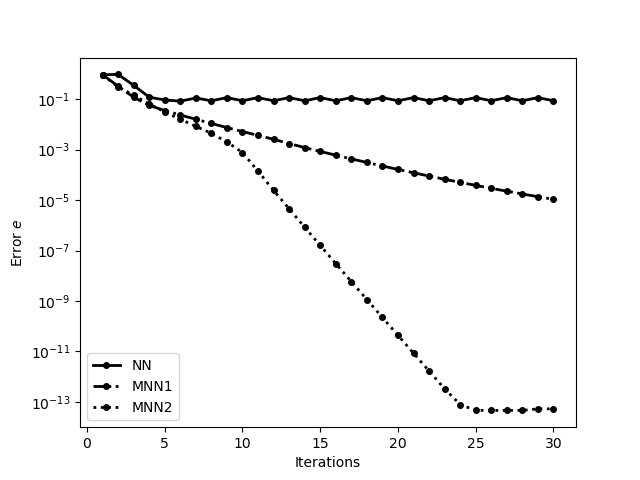}
  \caption{}
\label{fig:sub2}
\end{subfigure}%
\caption{\emph{Performance of the Neumann--Neumann method (NN) and its modifications (MNN1, MNN2). Figure (a) illustrates the errors compared to the total amount of linear solves, generated by Newton's method, at each iteration for the three methods when applied to the semilinear equation~\cref{eq:semi}. Here, NN requires roughly 50\% more solves, for small errors, compared to MNN2. Figure (b) shows the errors at each iteration for the three methods when applied to the $p$-Laplace equation~\cref{eq:3lap}. Note that the error for the NN approximation does not decrease after five iterations.}}
\label{fig:test}
\end{figure}%

Hence, there are two goals for this study. First, to provide numerical evidence of the advantages of the modified Neumann--Neumann methods for quasilinear equations, including degenerate ones. Secondly, to rigorously prove linear convergence for the modified methods when applied to semi- or quasilinear elliptic equations with domains and subdomains that are only assumed to be Lipschitz. 

The paper is organized as follows. In~\cref{sec:prel} we give the abstract framework to analyze nonlinear domain decomposition methods and state the assumptions on our equations. In~\cref{sec:weak} we then introduce the weak formulations and in~\cref{sec:tran} we define the transmission problem and the nonlinear Steklov--Poincaré operators. In~\cref{sec:methods} we describe the Steklov--Poincaré interpretations of the two modified Neumann--Neumann methods. We then prove convergence of the first modified method for nondegenerate quasilinear equations in~\cref{sec:conv1} and in~\cref{sec:conv2} we prove the convergence of the second modified method for semilinear equations. In~\cref{sec:discrete} we give a brief explanation how the same analysis can be applied to the Galerkin discretization of the elliptic equation. Finally, in~\cref{sec:numexp} we present some numerical results that verify our analysis.

To keep the notation simple we will denote generic constants by $c, C>0$, which may change from line to line.
    
\section{Preliminaries}\label{sec:prel}
Let $X$ be a Hilbert space and $X^*$ its dual. We use the notation $\langle\cdot, \cdot\rangle_{X^*\times X}$ for the dual pairing, or simply $\langle\cdot, \cdot\rangle$ if the spaces are obvious from context. A (nonlinear) operator $G: X\rightarrow X^*$ is said to be demicontinuous if $x^n\rightarrow x$ in $X$ implies that
    \begin{displaymath}
        \langle Gx^n-Gx, y\rangle\rightarrow 0\qquad\text{for all } y\in X
    \end{displaymath}
and $G$ is said to be coercive if
    \begin{displaymath}
        \lim_{\|x\|_X\rightarrow\infty}\frac{\langle Gx, x\rangle}{\|x\|_X}=\infty.
    \end{displaymath}
Moreover, the operator $G$ is uniformly monotone if there exists a $c>0$ such that
\begin{displaymath}
        \langle Gx-Gy, x-y\rangle\geq c\|x-y\|_X^2 \quad\text{ for all } x, y\in X.
\end{displaymath}
Note that a uniformly monotone operator $G$ is also coercive, since
\begin{displaymath}
\frac{\langle Gx, x\rangle}{\|x\|_X}=\frac{\langle Gx-G(0), x-0\rangle}{\|x-0\|_X}+\frac{\langle G(0), x\rangle}{\|x\|_X}\geq c\|x\|_X-\|G(0)\|_{X^*}.
\end{displaymath}

Finally, for a linear operator $P:X\rightarrow X^*$ we say that $P$ is symmetric if
    \begin{displaymath}
        \langle Px, y\rangle=\langle Py, x\rangle \quad\text{ for all } x, y\in X.
 \end{displaymath}
Also note that for a linear $P$ the uniform monotonicity is equivalent to
    \begin{displaymath}
        \langle Px, x\rangle\geq c\|x\|_X^2 \quad\text{ for all } x\in X.
 \end{displaymath}
For the analysis, we require the following geometric assumption.
\begin{assumption}\label{ass:domain}
        The domains $\Omega, \Omega_i\subset\R^d$ are bounded and Lipschitz. The interface $\Gamma$ and exterior boundaries $\partial\Omega\setminus\partial\Omega_i$ are $(d-1)$-dimensional Lipschitz manifolds.
    \end{assumption}
    We denote the standard Sobolev spaces by $H^1$ and $H^1_0$. We will also make use of the fractional Sobolev space $H^{1/2}$, see~\cite[Appendix A.2]{Widlund} for a proper definition. We denote our Sobolev spaces by
    \begin{align*}
        V&=H_0^1(\Omega),\quad V_i^0=H_0^1(\Omega_i),\\
        V_i&=\{u\in H^1(\Omega_i):\, T_iu=0\},\\
        \Lambda &=\{\mu\in L^2(\Gamma):\, E_i\mu\in H^{1/2}(\partial\Omega_i)\},
    \end{align*}
    with the usual norms for $V$ and $V_i^0$, and
    \begin{align*}
        \|u\|_{V_i}&=\|u\|_{L^2(\Omega_i)}+|u|_{V_i}=\|u\|_{L^2(\Omega_i)}+\|\nabla u\|_{L^2(\Omega_i)^d},\\
        \|\mu\|_{\Lambda} &=\|E_i\mu\|_{H^{1/2}(\partial\Omega_i)}.
    \end{align*} 

Here, $E_i: L^2(\Gamma)\rightarrow L^2(\partial\Omega_i)$ denotes the extension by zero and $T_i:V_i\rightarrow\Lambda$ denotes the trace operator on the interface $\Gamma\subset\partial\Omega_i$. The trace operator is a bounded linear operator~\cite[Lemma 4.4]{ehee22}. We will also make use of the fact that $T_i$ has a bounded linear right inverse, which we denote by $R_i:\Lambda\rightarrow V_i$. For more details, we refer to~\cite[Section 4]{ehee22}.

    We assume the following structure on the equation~\cref{eq:strong}.
    \begin{assumption}\label{ass:eq}
    The functions $x\mapsto \alpha(x, y, z)$ and $x\mapsto \beta(x, y, z)$ are measurable for almost all $y\in \R$ and $z\in \R^d$. Moreover, $\alpha$ and $\beta$ satisfy the following conditions, where $h_\ell$ are nonnegative functions in $L^\infty(\Omega)$.
    \begin{itemize}
        \item The function $\alpha$ is Lipschitz continuous in the variables $y$ and $z$. That is,
        \begin{displaymath}
        	|\alpha(x, y, z)-\alpha(x, y', z')|\leq h_1(x) \bigl(|z-z'|+|y-y'|\bigr)
        \end{displaymath}
        for all $y,y'\in\R,\, z,z'\in\R^d$.
        \item The function $\beta$ satisfies the local Lipschitz bound
        \begin{displaymath}
        	|\beta(x, y, z)-\beta(x', y', z')|\leq L(y, y')|y-y'|+h_1(x)|z-z'|
        \end{displaymath}
        for all $y,y'\in\R,\, z,z'\in\R^d$. Here we assume that the Lipschitz constant $L$ satisfies the growth bound
        \begin{equation}\label{eq:Lbound}
            L(y, y')\leq C\bigl(1+|y|^{p^*-2}+|y'|^{p^*-2}\bigr),
        \end{equation}
        where $p^*$ is the Sobolev conjugate of $p=2$, i.e.,
        \begin{displaymath}
            p^*=\frac{2d}{d-2},
        \end{displaymath}
        or if $d=2$, then we can take any $p^*\geq 2$.
        \item The functions $\alpha$ and $\beta$ satisfy the uniform monotonicity 
        \begin{align*}
        	&\bigl(\alpha(x, y, z)-\alpha(x, y', z')\bigr)\cdot(z-z')+\bigl(\beta(x, y, z)-\beta(x, y', z')\bigr)(y-y')\\
            &\quad\geq h_2(x)|z-z'|^2-h_3(x)|y-y'|^2
        \end{align*}
        for all $y,y'\in\R,\, z,z'\in\R^d$. Here,
        \begin{equation}\label{eq:monbound}
            \inf_{x\in\Omega} h_2(x) > C_p \sup_{x\in\Omega} h_3(x),
        \end{equation}
        where $C_p$ is the largest of the Poincaré constants of $\Omega$ and $\Omega_i$.
        \item The source term $f\in V^*$ and can be written
        \begin{displaymath}
            \langle f, v \rangle=\langle f_1, \restr{v}{\Omega_1}\rangle+\langle f_2, \restr{v}{\Omega_2}\rangle\qquad\text{for all } v\in V,
        \end{displaymath}
        with $f_i\in V_i^*$, $i=1, 2$.
    \end{itemize}
    \end{assumption}
    For sake of simplicity we will only derive convergence for our second method under the following restriction to semilinear equations.
    \begin{assumption}\label{ass:eqjacobian}
        The functions $\alpha$ and $\beta$ satisfy the following conditions, where $h_\ell$ are nonnegative functions in $L^\infty(\Omega)$.
        \begin{itemize}
        \item The function $\alpha$ is of the form $\alpha(x, y, z)=J_\alpha(x)z$ with $J_\alpha\in L^\infty(\Omega)^{d\times d}$. That is, $\alpha$ is linear in $z$.
        \item The function $\beta(x, y, z)=\beta(x, y)$ is differentiable with respect to $y$ and has measurable Jacobian $J_\beta:\Omega\times\R\rightarrow\R$, which satisfies the bounds
        \begin{align*}
        	|J_\beta(x, y)|&\leq h_1(x)+L(y),\\
        	|J_\beta(x, y)-J_\beta(x, y')|&\leq \Tilde{L}(y, y')|y-y'|
        \end{align*}
        for all $y,y'\in\R$. Here, we assume that the constants $L$ and $\Tilde{L}$ satisfies the growth bounds
        \begin{displaymath}
            L(y)\leq C\bigl(1+|y|^{p^*-2}\bigr)\quad\text{and} \quad\Tilde{L}(y, y')\leq C\bigl(1+|y|^{p^*-3}+|y'|^{p^*-3}\bigr),
        \end{displaymath}
        respectively, where $p^*$ is the Sobolev conjugate of $p=2$, i.e.,
        \begin{displaymath}
            p^*=\frac{2d}{d-2},
        \end{displaymath}
        or if $d=2$, then we can take any $p^*\geq 2$.
        \item The Jacobians $J_\alpha$ and $J_\beta$ satisfy the coercivity bound
        \begin{displaymath}
        J_\alpha(x)z'\cdot z'+J_\beta(x, y)(y')^2\geq h_2(x)|z'|^2-h_3(x)|y'|^2
        \end{displaymath}
        for all $y,y'\in\R,\, z,z'\in\R^d$, where $h_2,h_3$ are as in~\cref{eq:monbound}.
        \end{itemize}
    \end{assumption}
    \begin{example}\label{ex:semi}
        Consider the semilinear reaction-diffusion equation~\cref{eq:semi}, which corresponds to $\alpha(z)=z$ and $\beta(y)=|y|y$. This equation satisfies~\cref{ass:eq} with $h_1(x)=h_2(x)=1$, $h_3(x)=0$, and
        \begin{displaymath}
            L(y, y') = C\bigl(|y|+|y'|\bigr).
        \end{displaymath}
        The Lipschitz constant satisfies the growth bound for $d=2, 3$. Similarly, the equation satisfies~\cref{ass:eqjacobian}.
    \end{example}
    \begin{example}\label{ex:quasi}
        Consider the quasilinear equation with
        \begin{displaymath}
             \alpha(x, u, \nabla u)=\nabla u+\gamma(x)\sin(|\nabla u|)(1,\ldots,1)^\mathrm{T}
    \end{displaymath}
    and $\beta=0$. This equation satisfies~\cref{ass:eq} assuming that $\|\gamma\|_{L^\infty(\Omega)}$ is small enough, but will not satisfy~\cref{ass:eqjacobian} since $\alpha$ is not linear in $z$.
    \end{example}
    \begin{remark}
        Let $p\geq 2$ and consider the quasilinear equation~\cref{eq:3lap}, which arises as an implicit Euler step of the time dependent $p$-Laplace equation and corresponds to $\alpha(z)=|z|z$ and $\beta(y)=y$. This equation does not fulfill~\cref{ass:eq} or~\cref{ass:eqjacobian}, but will serve to illustrate the efficiency of our numerical schemes.
    \end{remark}
    \section{Weak formulations}\label{sec:weak}
    We define the operators $A:V\rightarrow V^*$ and $A_i: V_i\rightarrow V_i^*$ by
    \begin{align*}
        \langle Au, v\rangle_{}&=\int_{\Omega}\alpha(x, u, \nabla u)\cdot\nabla v+\beta(x, u, \nabla u) v \mathrm{d}x\qquad\text{and}\\
        \langle A_iu_i, v_i\rangle&=\int_{\Omega_i}\alpha(x, u_i, \nabla u_i)\cdot\nabla v_i+\beta(x, u_i, \nabla u_i) v_i \mathrm{d}x,
    \end{align*}
    respectively.
    \begin{lemma}\label{lemma:ai}
        Suppose that~\cref{ass:domain,ass:eq} hold. The operators $A_i$ then satisfy the Lipschitz condition
        \begin{equation*}
            \|A_iu-A_iv\|_{V_i^*}\leq L\bigl(\|u\|_{V_i},\|v\|_{V_i}\bigr)\|u-v\|_{V_i},
        \end{equation*}
        where $L$ satisfies the growth bound
        \begin{displaymath}
            L\bigl(\|u\|_{V_i},\|v\|_{V_i}\bigr)\leq C\bigl(1+\|u\|_{V_i}^{p^*-2}+\|v\|_{V_i}^{p^*-2}\bigr).
        \end{displaymath}
        Moreover, the operators $A_i$ are uniformly monotone.
    \end{lemma}
    \begin{proof}
        Let
        \begin{displaymath}
            q=\frac{p^*}{p^*-2},
        \end{displaymath}
        as in~\cref{ass:eq}. This implies that that
        \begin{displaymath}
            q^*=\frac{q}{q-1}=\frac{p^*}{2}.
        \end{displaymath}
        It follows that for $f\in L^q(\Omega_i)$ and $g, h\in L^{p^*}(\Omega_i)$ we have the three term Hölder inequality
        \begin{align*}
            \int_{\Omega_i}|fgh|\mathrm{d}x&\leq\|f\|_{L^q(\Omega_i)}\Bigl(\int_{\Omega_i}|gh|^{q^*}\mathrm{d}x\Bigr)^{1/q^*}\leq\|f\|_{L^q(\Omega_i)}\||g|^{q^*}\|_{L^2(\Omega_i)}^{1/q^*}\||h|^{q^*}\|_{L^2(\Omega_i)}^{1/q^*}\\
            &=\|f\|_{L^q(\Omega_i)}\|g\|_{L^{2q^*}(\Omega_i)}\|h\|_{L^{2q^*}(\Omega_i)}=\|f\|_{L^q(\Omega_i)}\|g\|_{L^{p^*}(\Omega_i)}\|h\|_{L^{p^*}(\Omega_i)}.
        \end{align*}
        Using this, the regular Hölder inequality, and the Lipschitz continuity of $\alpha$ and $\beta$ yields
        \begin{align*}
            &\bigl|\langle A_iu-A_iv, w\rangle\bigr|\\
            &\quad=\biggl|\int_{\Omega_i}\alpha(x, u, \nabla u)\cdot\nabla w-\alpha(x, v, \nabla v)\cdot\nabla w+\beta(x, u, \nabla u) w-\beta(x, v, \nabla v) w \mathrm{d}x\biggr|\\
            &\quad\leq\int_{\Omega_i}\bigl|\alpha(x, u, \nabla u)-\alpha(x, v, \nabla v)\bigr||\nabla w|+\bigl|\beta(x, u, \nabla u) -\beta(x, v, \nabla v)\bigr| |w| \mathrm{d}x\\
            &\quad\leq\int_{\Omega_i}h_1|\nabla u-\nabla v||\nabla w|+\bigl(L(u, v)|u -v|+h_1|\nabla u-\nabla v|\bigr) |w| \mathrm{d}x\\
            &\quad\leq C|u-v|_{V_i}\|w\|_{V_i}+\|L(u, v)\|_{L^q(\Omega_i)}\|u-v\|_{L^{p*}(\Omega_i)}\|w\|_{L^{p*}(\Omega_i)}.
        \end{align*}
        The Lipschitz constant can now be estimated as follows
        \begin{align*}
            \|L(u, v)\|_{L^q(\Omega_i)}&\leq C\|1+|u|^{p^*-2}+|v|^{p^*-2}\|_{L^q(\Omega_i)}\\
            &\leq C\bigl(\|1\|_{L^q(\Omega_i)}+\||u|^{p^*-2}\|_{L^q(\Omega_i)}+\||v|^{p^*-2}\|_{L^q(\Omega_i)}\bigr)\\
            &\leq C\bigl(1+\|u\|_{L^{p^*}(\Omega_i)}^{p^*-2}+\|v\|_{L^{p^*}(\Omega_i)}^{p^*-2}\bigr)
        \end{align*}
        and by the Sobolev embedding $V_i\hookrightarrow L^{p*}(\Omega_i)$ we have that
        \begin{align*}
            \bigl|\langle A_iu-A_iv, w\rangle\bigr|\leq C\bigl(1+\|u\|_{V_i}^{p^*-2}+\|v\|_{V_i}^{p^*-2}\bigr)\|u-v\|_{V_i}\|w\|_{V_i}.
        \end{align*}
        To prove uniform monotonicity we use the monotonicity of $\alpha,\beta$ and~\cref{eq:monbound} to get that
        \begin{align*}
            &\langle A_iu-A_iv, u-v\rangle\\
            &\quad=\int_{\Omega_i}\Bigl(\alpha(x, u, \nabla u)\cdot\nabla (u-v)-\alpha(x, v, \nabla v)\cdot\nabla (u-v)\\
            &\qquad+\beta(x, u, \nabla u) (u-v)-\beta(x, v, \nabla v) (u-v)\Bigr)\mathrm{d}x\\
            &\quad\geq\int_{\Omega_i}h_2(x)|\nabla u-\nabla v|^2-h_3(x)|u-v|^2 \mathrm{d}x\\
            &\quad\geq\inf_{x\in\Omega} h_2(x)\int_{\Omega_i}|\nabla u-\nabla v|^2\mathrm{d}x-\sup_{x\in\Omega_i} h_3(x)\int_{\Omega_i}|u-v|^2 \mathrm{d}x\\
            &\quad\geq\inf_{x\in\Omega_i} h_2(x)|u-v|_{V_i}^2-\sup_{x\in\Omega_i}h_3(x)\|u-v\|_{L^2(\Omega_i)}^2\\
            &\quad\geq c\|u-v\|_{V_i}^2.
        \end{align*}
    \end{proof}
    The weak formulation of the quasilinear elliptic equation is to find $u\in V$ such that $Au=f$ in $V^*$, or equivalently,
    \begin{equation}\label{eq:weak}
        \langle Au, v\rangle=\langle f, v\rangle\qquad \textrm{for all } v\in V.
    \end{equation}
    Similarly, the weak formulation of the inhomogeneous problem on $\Omega_i$ is the following: For $\eta\in\Lambda$ find $u_i\in V_i$ such that $T_iu_i=\eta$ and
    \begin{equation}\label{eq:weaki}
        \langle A_iu_i, v_i\rangle=\langle f, v_i\rangle \qquad\textrm{for all }v_i\in V_i^0.
    \end{equation}
    This problem has a unique solution and defines a solution operator
    \begin{displaymath}
        F_i:\Lambda\rightarrow V_i:\eta\mapsto u_i,
    \end{displaymath}
    see e.g.~\cite[Lemma 6.1]{ehee22}.
    \begin{lemma}\label{lemma:Fi}
        Suppose that~\cref{ass:domain,ass:eq} hold. The solution operators $F_i:\Lambda\rightarrow V_i$ then satisfy the following local Lipschitz bound
        \begin{displaymath}
            \|F_i\eta-F_i\mu\|_{V_i}\leq L\bigl(\|F_i\eta\|_{V_i}, \|F_i\mu\|_{V_i}\bigr) \|\eta-\mu\|_\Lambda,
        \end{displaymath}
        where $L$ has the following growth
        \begin{displaymath}
            L\bigl(\|u\|_{V_i},\|v\|_{V_i}\bigr)\leq C\bigl(1+\|u\|_{V_i}^{p^*-2}+\|v\|_{V_i}^{p^*-2}\bigr).
        \end{displaymath}
    \end{lemma}
    \begin{proof}
        Let $w_i=(R_i\eta-R_i\mu)-(F_i\eta-F_i\mu)$ and note that $T_iw_i=0$, which implies that $w_i\in V_i^0$. Therefore, by~\cref{eq:weaki,lemma:ai} we get
        \begin{align*}
            c\|F_i\eta-F_i\mu\|^2_{V_i}
	        &\leq \langle A_iF_i\eta-A_iF_i\mu, F_i\eta-F_i\mu\rangle\\
	        &= \langle A_iF_i\eta-A_iF_i\mu, R_i(\eta-\mu)\rangle- \langle A_iF_i\eta-A_iF_i\mu, w_i\rangle\\
	        &= \langle A_iF_i\eta-A_iF_i\mu, R_i(\eta-\mu)\rangle- \langle f_i-f_i, w_i\rangle\\
	        &\leq L\bigl(\|F_i\eta\|_{V_i}, \|F_i\mu\|_{V_i}\bigr) \|F_i\eta-F_i\mu\|_{V_i}\|R_i(\eta-\mu)\|_{V_i}\\
            &\leq L\bigl(\|F_i\eta\|_{V_i}, \|F_i\mu\|_{V_i}\bigr) \|F_i\eta-F_i\mu\|_{V_i}\|\eta-\mu\|_\Lambda.
        \end{align*}
        Dividing by $c\|F_i\eta-F_i\mu\|_{V_i}$ gives the Lipshchitz bound.
    \end{proof}
    \section{The transmission problem and the Steklov--Poincaré formulation}\label{sec:tran}
    The transmission problem is to find $(u_1, u_2)\in V_1\times V_2$ such that
\begin{equation}\label{eq:weaktran}
	\left\{\begin{aligned}
	     \langle A_iu_i, v_i\rangle&=\langle f_i, v_i\rangle & & \text{for all } v_i\in V_i^0,\, i=1,2,\\
	     T_1u_1&=T_2u_2, & &\\
	     \textstyle\sum_{i=1}^2 \langle A_i& u_i, R_i\mu\rangle-\langle f_i, R_i\mu\rangle=0 & &\text{for all }\mu\in \Lambda. 
	\end{aligned}\right.
\end{equation}
    The transmission problem is equivalent to the weak problem~\cref{eq:weak}, see~\cite[Theorem 5.2]{ehee22}. We define the nonlinear Steklov--Poincaré operators $S_i:\Lambda\rightarrow\Lambda^*$ by
    \begin{displaymath}
        \langle S_i\eta, \mu\rangle_{\Lambda^*\times\Lambda} =\langle A_iF_i\eta-f_i, R_i\mu\rangle_{V_i^*\times V_i}.
    \end{displaymath}
    \begin{remark}\label{rem:spind}
        We will repeatedly use that the Steklov--Poincaré operators are independent of the choice of extension $R_i$. To see this consider arbitrary extensions $R_i, \Tilde{R}_i$ such that $T_iR_i=T_i\Tilde{R}_i=I$ and let $w_i=R_i\mu-\Tilde{R}_i\mu$. Then $T_iw_i=0$ and therefore $w_i\in V_i^0$. It follows from the definition of $F_i$ that
        \begin{displaymath}
            \langle S_i\eta, \mu\rangle=\langle A_iF_i\eta-f_i, R_i\mu\rangle=\langle A_iF_i\eta-f_i, w_i\rangle+\langle A_iF_i\eta-f_i, \Tilde{R}_i\mu\rangle=\langle A_iF_i\eta-f_i, \Tilde{R}_i\mu\rangle.
        \end{displaymath}
    \end{remark}
    Writing out the definitions of the Steklov--Poincaré operators shows that the transmission problem can be reformulated as finding $\eta\in\Lambda$ such that
    \begin{equation}\label{eq:sp}
        S\eta=0\quad\textrm{in }\Lambda^*.
    \end{equation}
    The solution to the transmission problem can then be recovered as $(u_1, u_2)=(F_1\eta, F_2\eta)$. Conversely, the solution to the Steklov--Poincaré equation can be recovered from the solution to the transmission problem $(u_1, u_2)$ by setting $\eta=T_1u_1=T_2u_2$. For more details, see~\cite[Lemma 6.2]{ehee22}.

    The Steklov--Poincaré operators inherit the properties of the operators $A_i$.
    \begin{theorem}\label{thm:spop}
        Suppose that~\cref{ass:domain,ass:eq} hold. The Steklov--Poincaré operators $S_i:\Lambda\rightarrow \Lambda^*$ then satisfy the local Lipschitz condition
        \begin{align*}
            \| S_i\eta-S_i\mu\|_{\Lambda^*}&\leq L\bigl(\|\eta\|_\Lambda, \|\mu\|_\Lambda\bigr)^2\|\eta-\mu\|_\Lambda,
        \end{align*}
        where $L$ satisfies the growth bound
        \begin{displaymath}
            L\bigl(\|\eta\|_{\Lambda},\|\mu\|_{\Lambda}\bigr)\leq C\bigl(1+\|\eta\|_{\Lambda}^{p^*-2}+\|\mu\|_{\Lambda}^{p^*-2}\bigr).
        \end{displaymath}
        Moreover, $S_i$ is uniformly monotone. The same result holds for $S$.
    \end{theorem}
    \begin{proof}
        Let $L(y, y')$ denote any function satisfying~\cref{eq:Lbound}, with constant possibly changing line to line. To prove the Lipschitz bound we employ~\cref{lemma:ai,lemma:Fi} to get
        \begin{align*}
            &\bigl|\langle S_i\eta-S_i\mu, \lambda\rangle\bigr|=\bigl|\langle A_iF_i\eta-A_iF_i\mu, R_i\lambda\rangle\bigr|\\
            &\quad\leq L\bigl(\|F_i\eta\|_{V_i}, \|F_i\mu\|_{V_i}\bigr)\|F_i\eta-F_i\mu\|_{V_i}\|R_i\lambda\|_{V_i}\\
            &\quad\leq L\bigl(\|F_i\eta\|_{V_i}, \|F_i\mu\|_{V_i}\bigr)^2\|\eta-\mu\|_\Lambda\|R_i\lambda\|_{V_i}\\
            &\quad\leq L\bigl(\|\eta\|_\Lambda, \|\mu\|_\Lambda\bigr)^2\|\eta-\mu\|_\Lambda\|\lambda\|_\Lambda.
        \end{align*}
        The last inequality follows from the Lipschitz continuity of $F_i$ since it implies that
        \begin{displaymath}
            \|F_i\eta\|_{V_i}\leq C(1+\|\eta\|_\Lambda).
        \end{displaymath}
        For the uniform monotonicity, let $w_i=(R_i\eta-R_i\mu)-(F_i\eta-F_i\mu)$ and note that $T_iw_i=0$, which implies that $w_i\in V_i^0$. Therefore,
        \begin{align*}
            &\langle S_i\eta-S_i\mu, \eta-\mu\rangle=\langle A_iF_i\eta-A_iF_i\mu, R_i\eta-R_i\mu\rangle\\
            &\quad=\langle A_iF_i\eta-A_iF_i\mu, F_i\eta-F_i\mu\rangle+\langle A_iF_i\eta-A_iF_i\mu, w_i\rangle\\
            &\quad=\langle A_iF_i\eta-A_iF_i\mu, F_i\eta-F_i\mu\rangle\\
            &\quad\geq c\|F_i\eta-F_i\mu\|_{V_i}^2\geq c\|T_iF_i\eta-T_iF_i\mu\|_\Lambda^2\\
            &\quad= c\|\eta-\mu\|_\Lambda^2.
        \end{align*}
        The fact that the same holds for $S$ follows since $S=S_1+S_2$.
    \end{proof}
    \begin{corollary}\label{cor:spsol}
        Suppose that~\cref{ass:domain,ass:eq} hold. Then the Steklov--Poincaré operators $S_i, S$ are bijective. In particular, the Steklov--Poincaré equation~\cref{eq:sp} has a unique solution $\eta\in\Lambda$.
    \end{corollary}
    \begin{proof}
        We will use the Browder--Minty theorem~\cite[Theorem 26.A]{zeidler} in order to prove bijectivity and must therefore show that $S_i$ is demicontinuous, coercive, and monotone. Note that the uniform monotonicity derived in~\cref{thm:spop} directly yields that $S_i$ is coercive and monotone. Demicontinuity follows from the local Lipschitz bound in~\cref{thm:spop}. To see this, let $\eta^n\rightarrow \eta$ in $\Lambda$ and note that $\|\eta^n\|_{\Lambda}\leq C$. Therefore
        \begin{displaymath}
            \|S_i\eta^n-S_i\eta\|_{\Lambda^*}\leq L\bigl(\|\eta\|_\Lambda, \|\eta^n\|_\Lambda\bigr)^2\|\eta-\eta^n\|_\Lambda\leq C\|\eta-\eta^n\|_\Lambda\rightarrow 0,
\end{displaymath}
as $\eta^n$ tends to $\eta$. The same properties for $S$ follow since $S=S_1+S_2$. Hence, $S_i, S$ are all bijective.
    \end{proof}
    \section{Domain decomposition methods}\label{sec:methods}
    The domain decomposition methods can be equivalently formulated on the interface by writing $\eta^n=T_iu_i^{n+1}$, where $u_i^n$ approximates $u_i=\restr{u}{\Omega_i}$, the solution of~\cref{eq:weaktran}. The interface iterates $\eta^n$ then approximate the solution of~\cref{eq:sp}. The iterates of the domain decomposition methods can also be recovered as $u_i^{n+1}=F_i\eta^n$. The methods studied here have interface iterations of the form
    \begin{displaymath}
        \eta^{n+1}=\eta^n-sP^{-1}(S\eta^n).
    \end{displaymath}
    By choosing $P^{-1}=s_1S_1^{-1}+s_2S_2^{-1}$ for some parameters $s_1, s_2>0$ we get the standard nonlinear Neumann--Neumann method. We now propose modified Neumann--Neumann methods based on linear preconditioners $P$. We suggest two choices for the operator $P$ that correspond to the methods~\cref{eq:mnn1strong,eq:mnn2strong}.
    \begin{remark}
        The introduction of a third parameter $s$ is redundant since we could replace $s_1$ and $s_2$ by $\Tilde{s}_1=ss_1$ and $\Tilde{s}_2=ss_2$, respectively. However, this simplifies the notation of the proofs. 
    \end{remark} 
  
    \begin{method}\label{method:1}
    The iteration~\cref{eq:mnn1strong} is equivalent to the interface iteration
    \begin{equation}\label{eq:method1sp}
        \eta^{n+1}=\eta^n-s P^{-1}S\eta^n,
    \end{equation}
    with the operator $P:\Lambda\rightarrow \Lambda^*$ defined as
    \begin{displaymath}
        P^{-1}=s_1P_1^{-1}+s_2P_2^{-1}.
    \end{displaymath}
    Here, $s_1, s_2>0$ are method parameters and $P_i$ is the Steklov--Poincaré operator corresponding to the Laplace equation on $\Omega_i$. That is, if $\hat{F}_i:\eta\mapsto u$ denotes the solution operator to the problem
    \begin{equation*}
	\left\{\begin{aligned}
	     \langle\hat{A}_iu, v\rangle=\int_{\Omega_i} \nabla u\cdot\nabla v\mathrm{d}x&=0\quad\textrm{for all }v\in V_i^0,\\
	     T_iu&=\eta,
	\end{aligned}\right.
    \end{equation*}
    \end{method}
    then
    \begin{displaymath}
        \langle P_i\eta, \mu\rangle_{\Lambda^*\times\Lambda} =\langle \hat{A}_i\hat{F}_i\eta, R_i\mu\rangle_{V_i^*\times V_i}.
    \end{displaymath}
    
    \begin{method}\label{method:2}
    The iteration~\cref{eq:mnn2strong} is equivalent to the interface iteration
    \begin{equation}\label{eq:method2sp}
        \eta^{n+1}=\eta^n-s P(\eta^n)^{-1}S\eta^n,
    \end{equation}
    with the operator $P(\nu):\Lambda\rightarrow \Lambda^*$ defined as
    \begin{displaymath}
        P(\nu)^{-1}=s_1P_1(\nu)^{-1}+s_2P_2(\nu)^{-1}.
    \end{displaymath}
    Here, $s_1, s_2>0$ are method parameters and $P_i(\nu)$ is the Steklov--Poincaré operator corresponding to the linearization of~\cref{eq:weaki} at $w_i=F_i\nu$. That is, if $\hat{F}_i(\nu):\eta\mapsto u$ denotes the solution operator to the problem
    \begin{equation*}
	\left\{\begin{aligned}
	     \langle\hat{A}_i(w_i)u, v\rangle=\int_{\Omega_i} J_\alpha(w_i)\nabla u\cdot\nabla v
        +J_\beta(w_i)uv\mathrm{d}x&=0\quad\textrm{for all }v\in V_i^0,\\
	     T_iu&=\eta,
	\end{aligned}\right.
    \end{equation*}
    then
    \begin{displaymath}
        \langle P_i(\nu)\eta, \mu\rangle_{\Lambda^*\times\Lambda} =\bigl\langle \hat{A}_i\bigl(F_i(\nu)\bigr)\hat{F}_i(\nu)\eta, R_i\mu\bigr\rangle_{V_i^*\times V_i}.
    \end{displaymath}
    \end{method}
    \section{Convergence analysis for MNN1}\label{sec:conv1}
    We will now prove convergence of~\cref{method:1}. We first require the following lemma.
    \begin{lemma}\label{lemma:Pcoer}
        Let $P_1$ and $P_2$ be linear operators from a (real) Hilbert space $X$ into $X^*$ that are bounded, uniformly monotone, and symmetric. Then $P_1$, $P_2$, and
        \begin{displaymath}
            s_1P_1^{-1}+s_2P_2^{-1}
        \end{displaymath}
        are bijective. Moreover, the linear operator $P=(s_1P_1^{-1}+s_2P_2^{-1})^{-1}$ is also a bounded, uniformly monotone, and symmetric.
    \end{lemma}
    \begin{proof}
        First note that $P_1$, $P_2$, and $s_1P_1+s_2P_2$ all have bounded inverses by the Lax--Milgram lemma. We verify that $P_2(s_2P_1+s_1P_2)^{-1}P_1$ is a left inverse to $s_1P_1^{-1}+s_2P_2^{-1}$ by
        \begin{align*}
            &\bigl(P_2(s_2P_1+s_1P_2)^{-1}P_1\bigr)\bigl(s_1P_1^{-1}+s_2P_2^{-1}\bigr)\\
            &\quad=\bigl(P_2(s_2P_1+s_1P_2)^{-1}P_1\bigr)\bigl(P_1^{-1}(s_2P_1+s_1P_2)P_2^{-1}\bigr)=I
        \end{align*}
and similarly one can show that it is a right inverse. It follows immediately that the inverse, i.e., $P$, is bounded. To show that $P$ is uniformly monotone, let $\eta\in X$, $\sigma=P\eta$, and $\lambda_i=P_i^{-1}\sigma$. As $P_i$ are uniformly monotone and bounded, and $P^{-1}=s_1P_1^{-1}+s_2P_2^{-1}$ is bounded, one has
        \begin{align*}
            \langle P\eta, \eta\rangle&=\bigl\langle \sigma,P^{-1}\sigma\bigr\rangle=\bigl\langle \sigma,(s_1P_1^{-1}+s_2P_2^{-1})\sigma\bigr\rangle\\
            &=s_1\langle P_1\lambda_1,\lambda_1\rangle+s_2\langle P_2\lambda_2,\lambda_2\rangle\\
            &\geq c\bigl(\|\lambda_1\|_X^2+\|\lambda_2\|_X^2\bigr)\\
            &\geq c\bigl(\|P_1\lambda_1\|_{X^*}^2+\|P_2\lambda_2\|_{X^*}^2\bigr)=c\|\sigma\|_{X^*}^2\\
            &\geq c\|P^{-1}\sigma\|_X^2=c\|\eta\|_X^2.
\end{align*}
Next, since $P_i$ is symmetric, we have for all $\sigma,\rho\in X^*$ that
        \begin{displaymath}
            \langle \sigma, P_i^{-1}\rho\rangle=\langle P_iP_i^{-1}\sigma, P_i^{-1}\rho\rangle=\langle P_iP_i^{-1}\rho, P_i^{-1}\sigma\rangle=\langle \rho, P_i^{-1}\sigma\rangle.
        \end{displaymath}
        The fact that $P$ is symmetric follows from the above, as
        \begin{align*}
            \langle P\eta, \mu\rangle&=\bigl\langle P\eta, P^{-1}P\mu\bigr\rangle\\
            &=s_1\bigl\langle P\eta, P_1^{-1}P\mu\bigr\rangle+s_2\bigl\langle P\eta, P_2^{-1}P\mu\bigr\rangle\\
            &=s_1\bigl\langle P\mu, P_1^{-1}P\eta\bigr\rangle+s_2\bigl\langle P\mu, P_2^{-1}P\eta\bigr\rangle\\
            &=\bigl\langle P\mu,P^{-1}P\eta\rangle=\langle P\mu,\eta\rangle.
        \end{align*}
    \end{proof}
    \begin{theorem}\label{thm:conv1}
        Let $X$ be a (real) Hilbert space and $G: X\rightarrow X^*$ a nonlinear operator that is uniformly monotone and satisfies the local Lipschitz condition
        \begin{displaymath}
            \| G\mu-G\lambda\|_{X^*}\leq L(\|\mu\|_X, \|\lambda\|_X)\|\mu-\lambda\|_X,
        \end{displaymath}
        where $L(\|\mu\|_X, \|\lambda\|_X)$ is bounded for bounded $\|\mu\|_X, \|\lambda\|_X$. Then for every $\chi\in X^*$ there exists a unique solution $\eta\in X$ to $G\eta=\chi$.
        
        Moreover, let $P:X\rightarrow X^*$ be a linear, bounded, uniformly monotone, and symmetric operator
and let $s>0$ be small enough. The iteration
        \begin{align*}
            \eta^{n+1}=\eta^n+sP^{-1}(\chi-G\eta^n)
        \end{align*}
        then converges to $\eta$ and satisfies the linear error estimate
        \begin{displaymath}
        	   \|\eta^n-\eta\|_X\leq C L^n\|\eta^0-\eta\|_X
        \end{displaymath}
        for any $\eta^0\in X$.
    \end{theorem}
    \begin{proof}
        First observe that $G$ is a bijection by the Browder--Minty theorem; compare with the proof of~\cref{cor:spsol}. Hence, for every $\chi\in X^*$ there exists a unique solution $\eta\in X$ to $G\eta=\chi$. Next, we define the operator 
        \begin{displaymath}
        K\mu=\mu+sP^{-1}(\chi-G\mu)
        \end{displaymath}
        and write
        \begin{align*}
            &\langle PK\mu-PK\lambda, K\mu-K\lambda\rangle =\langle P\mu-P\lambda-s(G\mu-G\lambda),\mu-\lambda-sP^{-1}(G\mu-G\lambda)\rangle\\
            &\quad=\langle P\mu-P\lambda,\mu-\lambda\rangle+s^2\langle G\mu-G\lambda,P^{-1}(G\mu-G\lambda)\rangle\\
            &\qquad-s\Bigl(\langle P\mu-P\lambda,P^{-1}(G\mu-G\lambda)\rangle+\langle G\mu-G\lambda,\mu-\lambda\rangle\Bigr)\\
            &\quad=I_1+s^2I_2-sI_3.
        \end{align*}
        The second term $I_2$ is estimated using the continuity of $P^{-1}$ and $G$ and the uniform monotonicity of $P$, which yields that
        \begin{align*}
            I_2&=\langle G\mu-G\lambda,P^{-1}(G\mu-G\lambda)\rangle\\
            &\quad\leq L\bigl(\|\mu\|_X, \|\lambda\|_X\bigr) \|\mu-\lambda\|_X\|P^{-1}(G\mu-G\lambda)\|_X\\
            &\quad \leq C_1L\bigl(\|\mu\|_X, \|\lambda\|_X\bigr)^2\|\mu-\lambda\|_X^2.
        \end{align*}
        The third term $I_3$ is estimated using the symmetry of $P$ and the monotonicity of $G$, i.e.,
        \begin{align*}
            I_3=2\langle G\mu-G\lambda,\mu-\lambda\rangle\geq c\|\mu-\lambda\|_X^2.
        \end{align*}
        We define
        \begin{displaymath}
            (\mu, \lambda)_P=\langle P\mu, \lambda\rangle
        \end{displaymath}
        and note that this is an inner product which gives a norm $\|\cdot\|_P$ that is equivalent to $\|\cdot\|_X$. The latter follows as $P$ is bounded, uniformly monotone, and symmetric. Thus
        \begin{gather}
            \begin{aligned}\label{eq:Kcontr}
                \|K\mu-K\lambda\|_P&=\langle PK\mu-PK\lambda, K\mu-K\lambda\rangle\\
                &\leq \Bigl(1+s^2C_1L\bigl(\|\mu\|_X, \|\lambda\|_X\bigr)^2-sc\Bigr)\|\mu-\lambda\|_P^2.
            \end{aligned}
        \end{gather}
        For $r>0$ we define the ball
        \begin{displaymath}
            D_r=\{\mu\in X: \|\mu-\eta\|_P\leq r\}
        \end{displaymath}
        and then fix $r>0$ large enough that $\eta^0\in D_r$. Since $L\bigl(\|\mu\|_X, \|\lambda\|_X\bigr)$ is bounded on $D_r$ we can find $C_2(r)>0$ such that $C_1L(\|\mu\|_P, \|\lambda\|_P)^2\leq C_2(r)$ for all $\mu,\lambda\in D_r$. We can then choose $s>0$ small enough that $1+s^2C_2(r)-sc<1$. This implies that, for $\mu\in D_r$, we have
        \begin{displaymath}
            \|K\mu-\eta\|_P= \|K\mu-K\eta\|_P\leq (1+s^2C_2(r)-sc)\|\mu-\eta\|_P\leq r(1+s^2C_2(r)-sc)<r,
        \end{displaymath}
        which shows that $K\mu\in D_r$. Hence, by induction $\{\eta^n\}\subset D_r$ as $\eta^0\in D_r$. Moreover, according to~\cref{eq:Kcontr},
        \begin{displaymath}
            \|K\mu-K\lambda\|_P\leq L\|\mu-\lambda\|_P
        \end{displaymath}
        with $L<1$ for all $\mu, \lambda\in D_r$. Therefore, 
        \begin{align*}
            \|\eta^n-\eta\|_X&\leq C\|\eta^n-\eta\|_P=C\|K\eta^{n-1}-K\eta\|_P\leq CL\|\eta^{n-1}-\eta\|_P\\
            &\leq CL^n\|\eta^0-\eta\|_P\leq CL^n\|\eta^0-\eta\|_X,
        \end{align*}
        which tends to zero as $n$ tends to infinity.
    \end{proof}
Setting $(X,G,\chi)=(\Lambda,S,0)$ in this abstract result immediately yields the convergence of~\cref{method:1} for quasilinear problems, after noting that $P$ is bounded, uniformly monotone, and symmetric. Since $P_1$ and $P_2$ are the Steklov--Poincaré operators for the Laplace equation this follows from~\cite[Chapter 4]{quarteroni} and~\cref{lemma:Pcoer}. In particular, the symmetry follows since $P_i$ is independent of the choice of extension $R_i$ and can be written
\begin{displaymath}
    \langle P_i\eta, \mu\rangle =\langle \hat{A}_i\hat{F}_i\eta, R_i\mu\rangle =\langle \hat{A}_i\hat{F}_i\eta, \hat{F}_i\mu\rangle,
\end{displaymath}
compare with~\cref{rem:spind}. The fact that the method is well defined, i.e., the iteration has a unique solution at each step, is also a consequence of the above.
    \begin{corollary}\label{cor:mnn1conv}
        Let~\cref{ass:domain,ass:eq} hold and suppose that $s_1, s_2>0$ are small enough. Then the iterates $\eta^n$ of the interface iteration~\cref{eq:method1sp} converges linearly to the solution of~\cref{eq:sp} in $\Lambda$ for any $\eta^0\in\Lambda$. Moreover the iterates $(u_1^{n+1}, u_2^{n+1})=(F_1\eta^n, F_2\eta^n)$ of~\cref{method:1} converges linearly to the solution of~\cref{eq:weaktran} in $V_1\times V_2$.
    \end{corollary}
    \section{Convergence analysis for MNN2}\label{sec:conv2}
    We will now prove convergence of our second method for semilinear problems. We will use the following abstract convergence result.
    \begin{theorem}\label{thm:conv2}
        Let $X$ be a (real) Hilbert space $G: X\rightarrow X^*$ a nonlinear operator that is uniformly monotone and satisfies the local Lipschitz condition
        \begin{displaymath}
            \| G\mu-G\lambda\|_{X^*}\leq L(\|\mu\|_X, \|\lambda\|_X)\|\mu-\lambda\|_X,
        \end{displaymath}
        where $L(\|\mu\|_X, \|\lambda\|_X)$ is bounded for bounded $\|\mu\|_X, \|\lambda\|_X$. Then for every $\chi\in X^*$ there exists a unique solution $\eta\in X$ to $G\eta=\chi$.
        
        Moreover, for any $\nu\in X$ let $P(\nu):X\rightarrow X^*$ be a symmetric operator such that the family of operators $P(\cdot)$ satisfy
        \begin{equation*}
            \begin{aligned}
                \|P(\nu)\mu\|_{X^*}&\leq C_1(\nu)\|\mu\|_X\qquad &\text{for all }\nu,\mu\in X,\\
                \langle P(\nu)\mu, \mu\rangle&\geq c\|\mu\|_X^2\qquad &\text{for all }\nu,\mu\in X,\\
                \bigl\|\bigl(P(\nu)-P(\mu)\bigr)\lambda\bigr\|_{X^*}&\leq C_2\bigl(\|\nu\|_X, \|\mu\|_X\bigr) \|\nu-\mu\|_X\|\lambda\|_X\qquad &\text{for all }\nu,\mu,\lambda\in X,
            \end{aligned}
        \end{equation*}
        where $C_1(\|\mu\|_X), C_2(\|\nu\|_X, \|\mu\|_X)$ are bounded for bounded $\|\nu\|_X$, $\|\mu\|_X$. In particular, $P(\nu)$ is bounded and uniformly monotone for all $\nu\in X$. Assume also that $\eta^0$ is close enough to $\eta$ and that $s>0$ is small enough. Then the iteration
        \begin{align*}
            \eta^{n+1}=\eta^n+sP(\eta^n)^{-1}(\chi-G\eta^n)
        \end{align*}
        converges to $\eta$ and satisfies the linear error estimate
        \begin{displaymath}
        	   \|\eta^n-\eta\|_X\leq C L^n\|\eta^0-\eta\|_X.
        \end{displaymath}
    \end{theorem}
    \begin{proof}
        The fact that $G\eta=\chi$ has a unique solution has already been proven in~\cref{thm:conv2}. We define the operator $K(\nu)\mu=\mu+sP(\nu)^{-1}(\chi-G\mu)$ and compute
        \begin{align*}
            &\langle P(\nu)K(\nu)\mu-P(\nu)K(\nu)\lambda, K(\nu)\mu-K(\nu)\lambda\rangle\\
            &\quad=\bigl\langle P(\nu)\mu-P(\nu)\lambda-s(G\mu-G\lambda),\mu-\lambda-s\bigl(P(\nu)^{-1}(G\mu-G\lambda)\bigr)\bigr\rangle\\
            &\quad=\langle P(\nu)\mu-P(\nu)\lambda,\mu-\lambda\rangle+s^2\langle G\mu-G\lambda,P(\nu)^{-1}(G\mu-G\lambda)\rangle\\
            &\qquad-s\Bigl(\langle P(\nu)\mu-P(\nu)\lambda,P(\nu)^{-1}(G\mu-G\lambda)\rangle+\langle G\mu-G\lambda,\mu-\lambda\rangle\Bigr)\\
            &\quad=I_1+s^2I_2-sI_3.
        \end{align*}
        The second term $I_2$ is estimated using the continuity of $P(\nu)^{-1}$ and $G$ and the coercivity of $P$, which yields that
        \begin{align*}
            I_2&=\langle G\mu-G\lambda,P(\nu)^{-1}(G\mu-G\lambda)\rangle\\
            &\quad\leq L\bigl(\|\mu\|_X, \|\lambda\|_X\bigr) \|\mu-\lambda\|_X\|P(\nu)^{-1}(G\mu-G\lambda)\|_X\\
            &\quad\leq C_3L\bigl(\|\mu\|_X, \|\lambda\|_X\bigr)^2\|\mu-\lambda\|_X^2.
        \end{align*}
        The third term $I_3$ can be divided further using the symmetry of $P(\nu)$ and estimated using the monotonicity of $G$, i.e., 
        \begin{align*}
            I_3=2\langle G\mu-G\lambda,\mu-\lambda\rangle\geq c\|\mu-\lambda\|_X^2.
        \end{align*}
        
        We define the inner products
        \begin{displaymath}
            (\mu, \lambda)_{P(\nu)}=\langle P(\nu)\mu, \lambda\rangle
        \end{displaymath}
        and note that they yield norms $\|\cdot\|_{P(\nu)}$ that are equivalent to $\|\cdot\|_X$. Thus
        \begin{align*}
            \|K(\nu)\mu-K(\nu)\lambda\|_{P(\nu)}&=\langle P(\nu)K(\nu)\mu-P(\nu)K(\nu)\lambda, K\mu-K\lambda\rangle\\
            &\leq \Bigl(1+s^2C_3L\bigl(\|\mu\|_X, \|\lambda\|_X\bigr)^2-sc\Bigr)\|\mu-\lambda\|_{P(\nu)}^2.
        \end{align*}
        We define
        \begin{displaymath}
            D_r=\{\mu\in X: \|\mu-\eta\|_{P(\eta)}\leq r\}, 
        \end{displaymath}
        where $\eta$ is the unique solutionto $G\eta=\chi$. Let $R>0$. For $\mu,\lambda\in D_R$ we can find $C_4=C_4(R)>0$ such that $C_3L(\|\mu\|_X, \|\lambda\|_X)\leq C_4$ and then choose $s>0$ small enough such that $L:=1+s^2C_4-sc<1$. Moreover, we can find $C_5=C_5(R)>0$ such that for all $\mu\in X,\nu\in D_R$ we have
        \begin{displaymath}
            \frac{1}{C_5}\|\mu\|_X\leq\|\mu\|_{P(\nu)} \leq C_5\|\mu\|_X.
        \end{displaymath}
        Then we can choose $r$ small enough that $0<r<R$ and
        \begin{displaymath}
            \bigl\|\bigl(P(\eta)-P(\nu)\bigr)\mu\bigr\|_{X^*}\leq \frac{1-\sqrt{L}}{\sqrt{L}C_5^2}\|\mu\|_X.
        \end{displaymath}
        for $\nu\in D_r, \mu\in X$. It follows that for $\nu\in D_r,\mu\in X$ we have 
        \begin{align*}
            \|\mu\|_{P(\eta)}^2=\|\mu\|_{P(\nu)}^2+\langle P(\eta)\mu-P(\nu)\mu,\mu\rangle\leq \Bigl(1+\frac{1-\sqrt{L}}{\sqrt{L}}\Bigr)\|\mu\|_{P(\nu)}^2\leq\frac{1}{\sqrt{L}}\|\mu\|_{P(\nu)}^2,\\
            \|\mu\|_{P(\nu)}^2=\|\mu\|_{P(\eta)}^2+\langle P(\nu)\mu-P(\eta)\mu,\mu\rangle\leq \Bigl(1+\frac{1-\sqrt{L}}{\sqrt{L}}\Bigr)\|\mu\|_{P(\eta)}^2\leq\frac{1}{\sqrt{L}}\|\mu\|_{P(\eta)}^2.
        \end{align*}
        For $\mu,\nu\in D_r$ we also have that
        \begin{align*}
            \|K(\nu)\mu-\eta\|_{P(\eta)}&= \|K(\nu)\mu-K(\nu)\eta\|_{P(\eta)}\leq \frac{1}{L^{1/4}}\|K(\nu)\mu-K(\nu)\eta\|_{P(\nu)}\\
            &\leq \frac{1+s^2C_4-sc}{L^{1/4}}\|\mu-\eta\|_{P(\nu)}\leq \frac{1+s^2C_4-sc}{\sqrt{L}}\|\mu-\eta\|_{P(\eta)}\\
            &\leq  \sqrt{L}r<r,
        \end{align*}
        which shows that $K(\nu)\mu\in D_r$. Moreover, for all $\mu, \lambda,\nu\in D_r$
        \begin{align*}
            \|K(\nu)\mu-K(\nu)\lambda\|_{P(\eta)}&\leq \frac{1}{L^{1/4}}\|K(\nu)\mu-K(\nu)\lambda\|_{P(\nu)}\\
            &\leq L^{3/4}\|\mu-\lambda\|_{P(\nu)}\leq \sqrt{L}\|\mu-\lambda\|_{P(\eta)}
        \end{align*}
        with $L<1$. 
        
        Since $\eta^{n+1}=K(\eta^n)\eta^n$ and $\eta^0\in D_r$, as $\eta^0$ is assumed to be sufficiently close to~$\eta$, we have by induction that $\{\eta^n\}\subset D_r$. Moreover,
        \begin{align*}
            \|\eta^n-\eta\|_X&\leq C\|\eta^n-\eta\|_{P(\eta)}=C\|K(\eta^{n-1})\eta^{n-1}-K(\eta^{n-1})\eta\|_{P(\eta)}\\
            &\leq C\sqrt{L}\|\eta^{n-1}-\eta\|_{P(\eta)}\\
            &\leq CL^{n/2}\|\eta^0-\eta\|_{P(\eta)}\leq CL^{n/2}\|\eta^0-\eta\|_X,
        \end{align*}
        which tends to zero as $n$ tends to infinity.
    \end{proof}
    \begin{corollary}\label{cor:mnn2conv}
        Let~\cref{ass:domain,ass:eq,ass:eqjacobian} hold and suppose that $s_1, s_2>0$ are small enough and that $\eta^0$ is close enough to $\eta$, the solution of~\cref{eq:sp}. Then the iterates $\{\eta^n\}$ of the interface iteration~\cref{eq:method2sp} converges linearly to $\eta$ in $\Lambda$. Moreover the iterates $(u_1^{n+1}, u_2^{n+1})=(F_1\eta^n, F_2\eta^n)$ of~\cref{method:2} converges linearly to the solution of~\cref{eq:weaktran} in $V_1\times V_2$.
    \end{corollary}
    \begin{proof}
        The properties of $S, S_1, S_2$ are acquired in~\cref{thm:spop,cor:spsol} and it therefore only remains to show that the properties of $P(\nu)$ hold. We will show that the linear operator $\hat{A}_i(w):V_i\rightarrow V_i^*$
        \begin{displaymath}
            \langle\hat{A}_i(w)u, v\rangle=\int_{\Omega_i} J_\alpha(w)\nabla u\cdot\nabla v
            +J_\beta(w)uv\,\mathrm{d}x
        \end{displaymath}
        is bounded, locally Lipschitz continuous, and uniformly monotone independently of $w$.\\
        
        By~\cref{ass:eqjacobian} we have that
        \begin{align*}
            \bigl|\langle\hat{A}_i(w)u, v\rangle\bigr|&\leq\int_{\Omega_i} |J_\alpha(x)\nabla u\cdot\nabla v|
            +|J_\beta(x, w)uv|\mathrm{d}x\\
            &\leq \|J_\alpha\|_{L^\infty(\Omega_i)^{d\times d}}\|u\|_{V_i}\|v\|_{V_i}\\
            &\quad+\|J_\beta(\cdot, w)\|_{L^q(\Omega_i)}\|u\|_{L^{p*}(\Omega_i)}\|v\|_{L^{p*}(\Omega_i)}\\
            &\leq C\|u\|_{V_i}\|v\|_{V_i}+\|1+|w|^{p^*-2}\|_{L^q(\Omega_i)}\|u\|_{L^{p*}(\Omega_i)}\|v\|_{L^{p*}(\Omega_i)}\\
            &\leq C\|u\|_{V_i}\|v\|_{V_i}+C\bigl(1+\|w\|_{L^{p^*}(\Omega_i)}^{p^*-2}\bigr)\|u\|_{L^{p*}(\Omega_i)}\|v\|_{L^{p*}(\Omega_i)}.
        \end{align*}
        This follows by the same three term Hölder inequality as in the proof of~\cref{lemma:ai}. The boundedness then follows from the Sobolev embedding $V_i\hookrightarrow L^{p*}(\Omega_i)$. 
        
        Similarly, the Lipschitz continuity of $\hat{A}_i$ follows from~\cref{ass:eqjacobian} and the Sobolev embedding $V_i\hookrightarrow L^{p*}(\Omega_i)$, since
        \begin{align*}
            &\bigl|\langle\hat{A}_i(w)u-\hat{A}_i(\Tilde{w})u, v\rangle\bigr|\leq\int_{\Omega_i} \Bigl(|\bigl(J_\alpha(x)-J_\alpha(x)\bigr)\nabla u\cdot\nabla v|
            \\
            &\qquad+|\bigl(J_\beta(x, w)-J_\beta(x, \Tilde{w})\bigr)uv|\mathrm{d}x\Bigr)\\
            &\quad\leq \|J_\beta(\cdot, w)-J_\beta(\cdot, \Tilde{w})\|_{L^q(\Omega_i)}\|u\|_{L^{p*}(\Omega_i)}\|v\|_{L^{p*}(\Omega_i)}\\
            &\quad\leq \|\Tilde{L}(w,\Tilde{w})(w-\Tilde{w})\|_{L^q(\Omega_i)}\|u\|_{L^{p*}(\Omega_i)}\|v\|_{L^{p*}(\Omega_i)}\\
            &\quad\leq \|1+|w|^{p^*-3}+|\Tilde{w}|^{p^*-3}\|_{L^{q'}(\Omega_i)}\|w-\Tilde{w}\|_{L^{p^*}(\Omega_i)}\|u\|_{L^{p*}(\Omega_i)}\|v\|_{L^{p*}(\Omega_i)}\\
            &\quad\leq C(1+\|w\|_{L^{p^*}(\Omega_i)}^{p^*-3}+\|\Tilde{w}\|^{p^*-3}_{L^{p^*}(\Omega_i)})\|w-\Tilde{w}\|_{L^{p^*}(\Omega_i)}\|u\|_{L^{p^*}(\Omega_i)}\|v\|_{L^{p^*}(\Omega_i)}.
        \end{align*}
        Here, 
        \begin{displaymath}
            q'=\frac{p^*}{p^*-3}.
        \end{displaymath}
        Finally, the uniform monotonicity follows from the same property of $J_\alpha$ and $J_\beta$ in~\cref{ass:eqjacobian}, as
        \begin{align*}
            \langle\hat{A}_i(w)u, u\rangle&=\int_{\Omega_i} J_\alpha(x)\nabla u\cdot\nabla u
            +J_\beta(x, w)u^2\,\mathrm{d}x\\
            &\geq h_2(x)\int_{\Omega_i} |\nabla u|^2-h_3(x)|u|^2\mathrm{d}x\\
            &\geq\inf_{x\in\Omega} h_2(x)\int_{\Omega_i}|\nabla u|^2\mathrm{d}x-\sup_{x\in\Omega_i} h_3(x)\int_{\Omega_i}|u|^2 \mathrm{d}x\\
            &\geq\inf_{x\in\Omega_i} h_2(x)|u|_{V_i}^2-\sup_{x\in\Omega_i}h_3(x)\|u\|_{L^2(\Omega_i)}^2\\
            &\geq c\|u\|_{V_i}^2.
        \end{align*}
        
    We now identify
    \begin{displaymath}
        \langle P_i(\nu)\eta, \mu\rangle =\bigl\langle \hat{A}_i\bigl(F_i(\nu)\bigr)\hat{F}_i(\nu)\eta, R_i\mu\bigr\rangle =\bigl\langle \hat{A}_i\bigl(F_i(\nu)\bigr)\hat{F}_i(\nu)\eta, \hat{F}_i(\nu)\mu\bigr\rangle,
    \end{displaymath}
    compare with~\cref{rem:spind}. This then yields that $P_i(\nu)$ is bounded, locally Lipschitz continuous, uniformly monotone, and symmetric by arguing as in~\cref{thm:spop}.
    \end{proof}
    \begin{remark}
        The convergence result for~\cref{method:2} is somewhat weaker than the result for~\cref{method:1}, since we assume semilinearity of the equation and that the initial guess $\eta^0$ is close enough to the solution $\eta$. However, in numerical experiments the method performs better and we expect that these extra assumptions are unnecessary in practice.
    \end{remark}
    \section{Discrete nonlinear domain decomposition methods}\label{sec:discrete}
    Let $V_i^h\subset V_i$ and $\Lambda^h\subset\Lambda$ be finite dimensional subsets and let $V_i^{h, 0}=V_i^h\cap V_i^0$. Then we can define the discrete operators
    $A^h:V^h\rightarrow (V^h)^*$ and $A_i^h: V_i^h\rightarrow (V_i^h)^*$ by
    \begin{align*}
        \langle A^hu, v\rangle&=\int_{\Omega}\alpha(x, u, \nabla u)\cdot\nabla v+\beta(x, u, \nabla u) v \mathrm{d}x\qquad\text{and}\\
        \langle A_i^hu_i, v_i\rangle&=\int_{\Omega_i}\alpha(x, u_i, \nabla u_i)\cdot\nabla v_i+\beta(x, u_i, \nabla u_i) v_i \mathrm{d}x,
    \end{align*}
    respectively. We assume the following compatibility relation between the spaces $V_i^h, \Lambda^h$.
    \begin{assumption}\label{ass:discrete}
        The trace operators $T_i$ map $V^h_i$ into $\Lambda^h$ and have bounded linear right inverses $R_i^h: \Lambda^h\rightarrow V_i^h$. We assume that $R_i^h$ is bounded independently of $h$.
    \end{assumption}
    Typically we would have that $V_1^h, V_2^h, \Lambda^h$ share the same grid points on the interface. In this case choosing $R_i^h$ to be the discrete harmonic extension gives $h$-independent bounds under the assumption that $V_i^h$ are finite element spaces on a regular family of triangulations, see~\cite[Theorem 4.1.3]{quarteroni}.

    The weak discrete problem is then to find $u^h\in V^h$ such that
    \begin{equation}\label{eq:weakh}
        \langle A^hu^h, v\rangle=\langle f, v\rangle \quad\textrm{for all }v\in V^h.
    \end{equation}
    Moreover, the discrete transmission problem is to find $(u_1^h, u_2^h)\in V_1^h\times V_2^h$ such that
    \begin{equation}\label{eq:weaktranh}
    	\left\{\begin{aligned}
    	     \langle A_i^hu_i^h, v_i\rangle&=\langle f_i, v_i\rangle & & \text{for all } v_i\in V_i^{h, 0},\, i=1,2,\\
    	     T_1u_1&=T_2u_2, & &\\
    	     \textstyle\sum_{i=1}^2 \langle A_i^h& u_i^h, R_i^h\mu\rangle-\langle f_i, R_i^h\mu\rangle=0 & &\text{for all }\mu\in \Lambda. 
    	\end{aligned}\right.
    \end{equation}
    As for the continuous case, the weak problem and the transmission problem are equivalent with $u_i^h=\restr{u^h}{\Omega_i}$. Under~\cref{ass:discrete}, there exists discrete solution operators $F_i^h:\Lambda^h\rightarrow V_i^h$ that satisfy $T_iF_i^h\eta=\eta$ and
    \begin{displaymath}
        \langle A_i^hF_i^h\eta, v\rangle=\langle f_i, v\rangle \quad\textrm{for all }v\in V_i^{h, 0}.
    \end{displaymath}
    Thus, we can define discrete nonlinear Steklov--Poincaré operators $S_i:\Lambda^h\rightarrow (\Lambda^h)^*$ as
    \begin{displaymath}
        \langle S_i^h\eta, \mu\rangle =\langle A_i^hF_i^h\eta, R_i^h\mu\rangle-\langle f_i, R_i^h\mu\rangle.
    \end{displaymath}
    Defining $S^h=S_1^h+S_2^h:\Lambda^h\rightarrow (\Lambda^h)^*$ we can reformulate the discrete transmission problem as finding $\eta_h\in\Lambda^h$ such that
    \begin{equation}\label{eq:sph}
        S_i^h\eta_h=0\quad\text{in }(\Lambda^h)^*.
    \end{equation}
    
    \begin{remark}
        The discrete Steklov--Poincaré operators are independent of the choice of extension $R_i^h$, i.e., for any $\Tilde{R}_i^h:\Lambda^h\rightarrow V_i^h$ such that $T_i\Tilde{R}_i^h=I$ we have
        \begin{displaymath}
            \langle S_i^h\eta, \mu\rangle =\langle A_i^hF_i^h\eta, R_i^h\mu\rangle-\langle f_i, R_i^h\mu\rangle=\langle A_i^hF_i^h\eta, \Tilde{R}_i^h\mu\rangle-\langle f_i, \Tilde{R}_i^h\mu\rangle,
        \end{displaymath}
        compare with~\cref{rem:spind}. In particular, this means that, in implementation, we can take $R_i^h$ as the extension by zero to all interior grid points, since this is more efficient to compute.
    \end{remark}
    
    The discrete nonlinear Steklov--Poincaré operators have the following properties. They follow by the same argument as in~\cref{thm:spop}.
    \begin{theorem}
        Let~\cref{ass:domain,ass:eq,ass:discrete} hold. Then the discrete Steklov--Poincaré operators $S_i^h:\Lambda^h\rightarrow (\Lambda^h)^*$ satisfy the Lipschitz bound
        \begin{align*}
            \| S_i^h\eta-S_i^h\mu\|_{(\Lambda^h)^*}&\leq L\bigl(\|\eta\|_\Lambda, \|\mu\|_\Lambda\bigr)^2\|\eta-\mu\|_\Lambda,
        \end{align*}
        where $L$ has the following growth
        \begin{displaymath}
            L\bigl(\|u\|_{\Lambda},\|v\|_{\Lambda}\bigr)\leq C\bigl(1+\|u\|_{\Lambda}^{p^*-2}+\|v\|_{\Lambda}^{p^*-2}\bigr).
        \end{displaymath}
        Moreover, $S_i^h$ are uniformly monotone, i.e.,
        \begin{displaymath}
            \langle S_i^h\eta-S_i^h\mu, \eta-\mu\rangle\geq c\|\eta-\mu\|_\Lambda^2.
        \end{displaymath}
        The constants $c, C, L$ are independent on the choice of the spaces $V_i^h, \Lambda^h$. The same result holds for $S^h$. In particular, there exists a unique solution $\eta_h\in\Lambda^h$ of~\cref{eq:sph}.
    \end{theorem}
    The Steklov--Poincaré interpretations of the discrete versions of~\cref{method:1,method:2} are
    \begin{align*}
        \eta_h^{n+1}=\eta_h^n-s (P^h)^{-1}S^h\eta_h^n\qquad\text{and}\\
        \eta_h^{n+1}=\eta_h^n-s (P^h(\eta_h^n))^{-1}S^h\eta_h^n,
    \end{align*}
    respectively. Here $P^h$ and $P^h(\nu)$ are the discrete versions of $P$ in~\cref{method:1,method:2}. The following convergence result follows directly from~\cref{thm:conv2} with the choice $(X, G, P)=(\Lambda^h, S^h, P^h)$.
    \begin{corollary}\label{cor:convh}
        Let~\cref{ass:domain,ass:eq,ass:discrete} hold and suppose that $s_1, s_2>0$ are small enough. Then the discrete version of~\cref{method:1} converges linearly in $V_1^h\times V_2^h$ to the solution of~\cref{eq:weaktranh}. The convergence factor $L$ is independent of the spaces $V_i^h, \Lambda^h$. If~\cref{ass:eqjacobian} also holds and the initial guess $\eta_h^0$ is close enough to $\eta_h$, the same result is true for~\cref{method:2}.
    \end{corollary}
    
    \section{Numerical results}\label{sec:numexp}
    \begin{figure}
        \centering
        \includegraphics[width=0.4\linewidth]{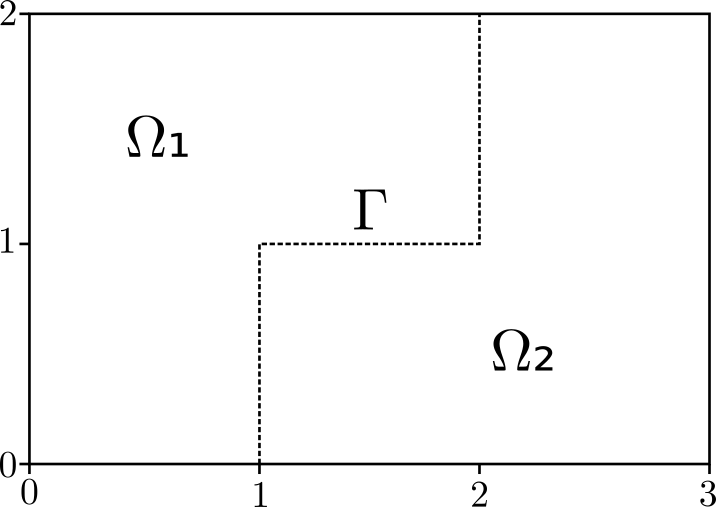}\qquad
        \includegraphics[width=0.4\linewidth]{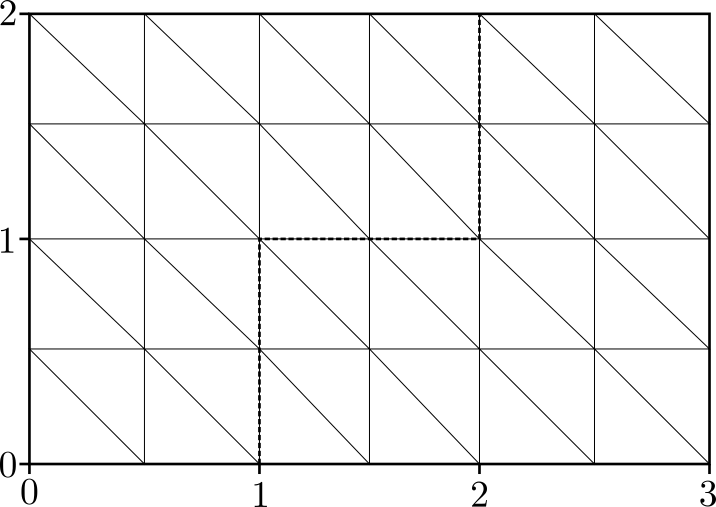}
        \caption{\emph{The domain decomposition used for all numerical results (left) and the computational mesh with $h=0.5$ (right).}}
        \label{fig:discretedomain}
    \end{figure}%
    For our numerical results, we consider the domain
    \begin{align*}
        \Omega&=[0, 3]\times [0, 2]\subset\R^2
    \end{align*}
    and the decomposition shown in~\cref{fig:discretedomain}. The domain is discretized using a finite element method with linear elements and the mesh width $h=1/256$. The grid is aligned with the interface $\Gamma$, so that~\cref{ass:discrete} is fulfilled. We first consider the semilinear equation as in~\cref{ex:semi}, which leads to Steklov--Poincaré operators with the properties of~\cref{thm:spop} and thus linearly convergent methods as in~\cref{cor:mnn1conv,cor:mnn2conv,cor:convh}. The source term is chosen as
    \begin{displaymath}
        f(x, y)=xy(3-x)(2-y)
    \end{displaymath}
    and the initial guess is $\eta^0=0$. At each iteration we compute the error
    \begin{displaymath}
        e=\frac{\|u_1^{n, h}-u_1^h\|_{V_1}+\|u_2^{n, h}-u_2^h\|_{V_2}}{\|u_1^h\|_{V_1}+\|u_2^h\|_{V_2}}.
    \end{displaymath}
    The error at each iteration is plotted in~\cref{fig:prelreslip}. Here, $u_i^h$ denotes the solution to the discrete problem~\cref{eq:weakh} restricted to $\Omega_i$ and $u_i^{n, h}$ denotes the $n$th iteration of the respective method. The method parameters are $s_1=s_2=0.2$ for the Neumann--Neumann method, $s_1=s_2=0.19$ for~\cref{method:1}, and $s_1=s_2=0.21$ for~\cref{method:2}. The method parameters are chosen to be near optimal. Since the modified methods are cheaper to compute than the Neumann--Neumann method we also plot the error compared to the total amount of linear systems solved to achieve this error. This can also be found in~\cref{fig:prelreslip}. From the graphs we see that the convergence rates are similar for the three methods, but the modified methods are more efficient, since they require less linear solves to achieve the same error.

We perform the same experiment with the same parameters on~\cref{ex:quasi} and plot the results in~\cref{fig:quasilip}. Now we see that the convergence rate for~\cref{method:2} is better than for~\cref{method:1} and in terms of linear solves, both methods require less linear solves to achieve similar error to the Neumann--Neumann method. Note that since~\cref{ex:quasi} satisfies~\cref{ass:eq}, but not~\cref{ass:eqjacobian} we have only proven convergence for~\cref{method:1}.
    
    Finally, the experiment is applied to the $p$-Laplace equation~\cref{eq:3lap} with $p=3$, although the equation does not satisfy~\cref{ass:eq}. The domain, mesh and source term is the same as before. The method parameters are $s_1=s_2=0.2$ for the Neumann--Neumann method, $s_1=s_2=0.15$ for~\cref{method:1}, and $s_1=s_2=0.2$ for~\cref{method:2}. The error is plotted in~\cref{fig:quasi}. We see that the Neumann--Neumann method does not converge while~\cref{method:1,method:2} converges linearly, with the latter converging faster, especially when the error is small.
    
\begin{figure}
    \centering
    \includegraphics[width=0.49\linewidth]{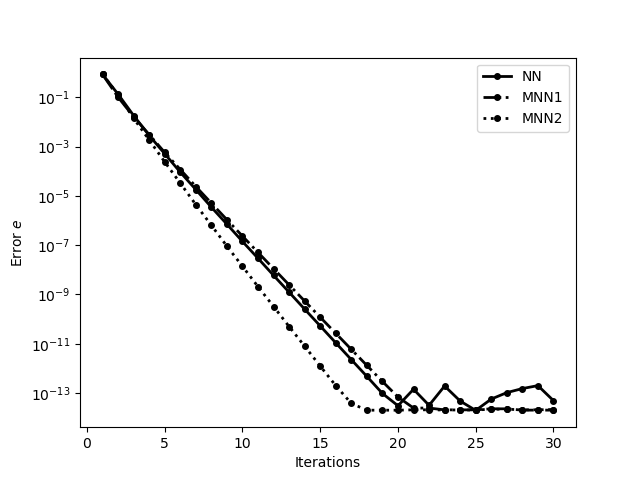}
    \includegraphics[width=0.49\linewidth]{nlelips020s019s021res8solves.png}
    \caption{The error compared to the amount of iterations (left) and the amount of linear solves (right) of the Neumann--Neumann (NN) and the modified Neumann--Neumann methods (MNN1, MNN2) applied to~\cref{ex:semi}.}
    \label{fig:prelreslip}
    \end{figure}%
\begin{figure}
    \centering
    \includegraphics[width=0.49\linewidth]{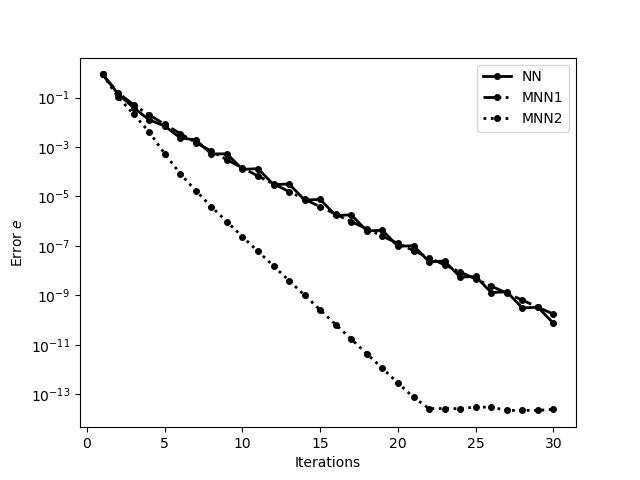}
    \includegraphics[width=0.49\linewidth]{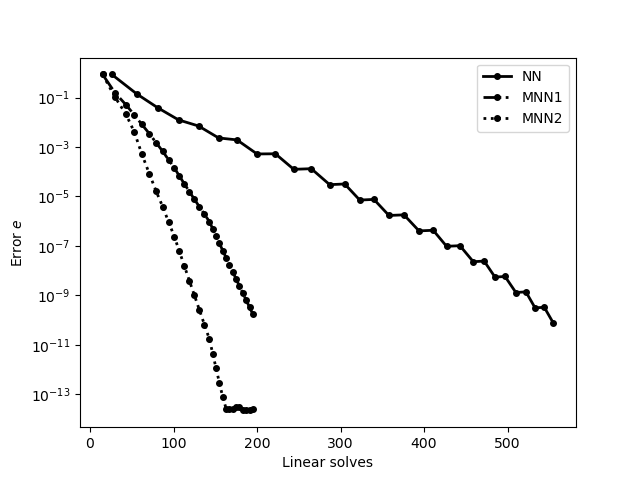}
    \caption{The error compared to the amount of iterations (left) and the amount of linear solves (right) of the Neumann--Neumann (NN) and the modified Neumann--Neumann methods (MNN1, MNN2) applied to~\cref{ex:quasi}.}
    \label{fig:quasilip}
\end{figure}%
\begin{figure}
    \centering
    \includegraphics[width=0.49\linewidth]{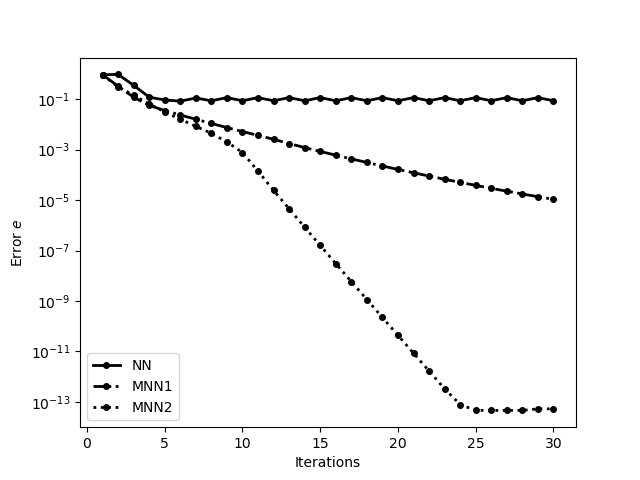}
    \caption{The error compared to the amount of iterations of the Neumann--Neumann (NN) and the modified Neumann--Neumann methods (MNN1, MNN2) applied to~\cref{eq:3lap}.}
    \label{fig:quasi}
\end{figure}%

\bibliographystyle{plain}
\bibliography{ref}
\end{document}